\documentclass[11pt]{amsart}
\usepackage{amsmath,amsfonts,amssymb,amsthm,amscd,latexsym,euscript,verbatim, bbold}
\usepackage[all]{xy}

\makeatletter
\def\section{\@startsection{section}{1}%
  \z@{1.1\linespacing\@plus\linespacing}{.8\linespacing}%
  {\normalfont\Large\scshape\centering}}
\makeatother

\theoremstyle{plain}

\newtheorem*{conj*}{Root Groups Conjecture}
\newtheorem*{thm1.2}{(1.2) Theorem}
\newtheorem*{thm1.3}{(1.3) Theorem}
\newtheorem*{thm1.4}{(1.4) Theorem}
\newtheorem*{prop*}{Proposition}

\newtheorem{prop}{Proposition}[section]

\newtheorem{thm}[prop]{Theorem}

\newtheorem{lemma}[prop]{Lemma}

\theoremstyle{definition}

\newtheorem*{Def*}{Definition}
\newtheorem{Defs}[prop]{Definitions}

\newtheorem{notation}[prop]{Notation}
\newtheorem*{notation*}{Notation}

\newtheorem{remarks}[prop]{Remarks}

\newcommand{\calh}{\mathcal{H}}

\newcommand{\caln}{\mathcal{N}}

\newcommand{\la}{\lambda}

\newcommand{\ff}{\mathbb{F}}

\newcommand{\kk}{\mathbb{K}}

\newcommand{\zz}{\mathbb{Z}}

\newcommand{\ga}{\alpha}
\newcommand{\gb}{\beta}
\newcommand{\gc}{\gamma}

\newcommand{\gd}{\delta}

\newcommand{\gl}{\lambda}

\newcommand{\gs}{\sigma}

\newcommand{\cent}{\operatorname{Cent}}

\newcommand{\lan}{\langle}
\newcommand{\ran}{\rangle}

\newcommand{\half}{\textstyle{\frac{1}{2}}}

\newcommand{\widebar}[1]{\overset{\mskip1mu\hrulefill\mskip1mu}{#1}
                \vphantom{#1}}

\newcommand{\e}{\mathbb{1}}

\numberwithin{equation}{section}

\begin{document}
\title[]{Associative and Jordan algebras generated by two idempotents}
\author[Louis Rowen, Yoav Segev]
{Louis Rowen$^1$\qquad Yoav Segev}

\address{Louis Rowen\\
         Department of Mathematics\\
         Bar-Ilan University\\
         Ramat Gan\\
         Israel}
\email{rowen@math.biu.ac.il}
\thanks{$^1$Partially supported by the Israel Science Foundation grant no.~1623/16}
\address{Yoav Segev \\
         Department of Mathematics \\
         Ben-Gurion University \\
         Beer-Sheva 84105 \\
         Israel}
\email{yoavs@math.bgu.ac.il}


\begin{abstract}
The purpose of this note is to obtain precise information
about  associative or Jordan algebras generated
by two idempotents.
\end{abstract}

\date{\today}
\maketitle
\section{Introduction}

This paper is motivated by recent work on commutative nonassociative algebras
generated by idempotents.  Such algebras  are for example
the Griess algebras associated with vertex operator algebras,
and Majorana algebras \cite{ivanov, matsuo, miyamoto, sakuma}.  Also Jordan algebras
generated by idempotents as well as Axial algebras (\cite{HRS1, HRS2, HSS}) are such,
see also \cite{DeMR}.

In these algebras the adjoint operator associated to the generating
idempotents (i.e., multiplication by the idempotent) is semi-simple
and has few eigenvalues.  Further, certain fusion rules
(i.e., multiplication rules), between the eigenspaces are assumed
(similar to the Peirce decomposition multiplication rules in Jordan algebras, see e.g.,
\cite[Theorem 4, p.~334]{ZSSS}).
One can then associate an involutive automorphism of the algebra to each of these
idempotents, and the group generated by these involutions
are sometimes of great interest (e.g., the Monster group).

In general, it is not unintuitive to think about idempotents in
these algebras in a similar way one thinks of involutions in a
group. In all these algebras it is important to know the subalgebras
generated by two idempotents. Some papers dealt with this question
in the associative case (e.g.~\cite{B,L, V}), and some in the Jordan
algebra case (e.g.~\cite{HRS2, sakuma}).

Throughout $\ff$ is a unital commutative ring.
Let $A$ be a (linear)
algebra over $\ff$ with multiplication denoted by $u\circ v,\ u,v\in
A,$ so if $A$ is a ring we take $\ff=\zz,$ the integers.
We let $A^{(1)}$ be the algebra $A$ if $A$ is unital
(i.e.~$A$ has an identity element), and $A^{(1)}=\ff\oplus A$ with multiplication
defined by
\[
(\ga,x)(\gb,y)=(\ga\gb,\ga y+\gb x+x\circ y),
\]
if $A$ does not have an identity element.  In the latter case $A^{(1)}$
has the identity element $\e=(1,0)$.  We identify $A$ with the subset
 $\{(0,x)\mid x\in A\}$ of $A^{(1)}$.  Thus $\e$ denotes the identity
element of $A^{(1)}$ (also in the case where $A$ is unital).

For an element $x\in A$ we let $x^0=\e\in A^{(1)}$. The cases that
will interest us in this note are the case where $A$ is associative,
and the case where $A=J$ is a Jordan algebra, and $\ff$ is a field
of characteristic not $2$.  In both cases, $A$ is {\it power
associative}, and we let, as usual,  $\ff[x]\subseteq A^{(1)}$ be
the subalgebra of~$A^{(1)}$ generated by $x$ over $\ff,$ i.e., the
set of polynomials in $x$ with coefficients in~$\ff$.

Some parts of our first theorem are mostly known in principle:

\begin{thm}\label{thm assoc}
Let $A$ be an associative algebra (not necessarily with $\e$)
over a unital commutative ring $\ff,$ generated by two distinct idempotents $a$ and $b$.  Denote
multiplication in~$A$ by {\bf juxtaposition}: $x y$.  Then
\begin{enumerate}
\item $($\cite[\S12.2]{B}, \cite[Lemma 3]{L}$)$
$\gs:=(a-b)^2$ is in the center of $A.$

\item
$A$ is spanned by $\gs, a, b$ and $a b$ as a module over
$\ff[\gs]$. In particular $A$ satisfies a multilinear polynomial identity.

\item
If $\gs=\e,$ then $a b a=b a b=0$.  Hence either $a b=b a=0$
and $A=\ff a\oplus\ff b,$ or one of $\ff(a b)$ or $\ff(b a)$ is a nontrivial square-zero ideal of $A$.

\item
If $\gs=0,$ then one of $\ff(a (b-\e))+\ff(b (a-\e))$ or $\ff(a-b)$
is a nontrivial square-zero ideal of $A$.

\item
In both  cases (3) and (4), $A$ has a nilpotent  ideal $I$
such that $A/I$ is commutative.  In case (3) we can take $I^2=0,$
and in case (4) we can take $I^3=0$.

\item
If  $\gs-\gs^2$ is invertible in $A$ (so that $\e\in A$), then
$A\cong M_2(\ff[\gs])$.  In particular,
if $\ff[\gs]$ is a field (so that $\ff\e$ is a field and $\gs$ is algebraic over $\ff\e$),
with $\gs \ne 0,\e,$ then  $A\cong M_2(\ff[\gs]).$

\item $($Compare with \cite[Theorem 4]{L}.$)$
If $A$ is simple, then $\e\in A$ and $\ff[\gs]$ is a field (so $\ff\e$
is a field and $\gs$ is algebraic over $\ff\e$).  Further,  either $A=\ff\e$ is a field
(and $\{a,b\}=\{0,\e\}$)  or  $A\cong M_2(\ff[\gs])$.

\item
Let $J=\ff[\gs]\gs+\ff[\gs]a+\ff[\gs]b$.  Then there is an involution $*$ on $A$ defined by:
\[
(\ga_{\gs}\gs +\ga_a a+\ga_b b+\ga_{a b}(a b))^*=\ga_{\gs}\gs +\ga_a a+\ga_b b+\ga_{a b}(b a)
\]
if and only if
\[\tag{$i$}
\ga (a b)\in J,\text{ for some }\ga\in\ff[\gs] \implies
\ga(b a-a b)=0.
\]
In particular, if $A$ is simple and not commutative,   then $*$ is an involution   on $A$.
\end{enumerate}
\end{thm}

\noindent
(See also Theorem \ref{thm assocmain} for additional significant information.)

For the notion of the
{\it center} of a Jordan algebra, see Definition \ref{defs main}(5)
below.
\begin{thm}\label{thm jordan}
Let $J$ be a Jordan algebra over a field $\ff$ of characteristic not $2$
generated by two distinct idempotents $a$ and $b$. Denote the multiplication
in $J$ by {\bf dot}: $x\cdot y$.  Then
\begin{enumerate}
\item
$\gs:=(a-b)^2$ is in the center of $J;$

\item
$J$ is spanned by $a, b,$ and $\gs$ as a module over $\ff[\gs];$

\item
if the Jordan algebra $J$ is simple, then either $J=\ff$ or
$J\cong\calh(A,*)$  the set of symmetric elements $x^*=x,$ where $A$
is a {\bf simple} algebra as in Theorem \ref{thm assoc}$(7)$, and
$*$ is as in Theorem \ref{thm assoc}$(8)$.
 $($Of course $\calh(A,*)$ is a
subalgebra of $A^+)$.
\end{enumerate}
\end{thm}

\section{Proofs of Theorem \ref{thm assoc} and Theorem \ref{thm jordan}}

Before we prove Theorems \ref{thm assoc} and   \ref{thm jordan}
we need a few definitions, the statement of the Shirshov-Cohn Theorem,
and a few lemmas.

\begin{Defs}\label{defs main}
\begin{enumerate}
\item
A (linear) algebra is just an algebra over a commutative unital ring $\ff$ in the usual sense
(but not necessarily associative).

\item
For an algebra $(A,\circ),$ the {\it commutator} is $[x,y]:=x\circ y-y\circ x$ and the {\it associator}
is $[x,y,z]:= (x\circ y)\circ z - x\circ (y\circ z).$

\item
The {\it nucleus} of an algebra is the part that associates with
everything, consisting of the elements associating in all possible
ways with all other elements:
\[
\caln uc(A) := \{x\in A \mid [x,A,A] = [A, x,A] = [A,A, x] = 0\}.
\]
\item
The {\it center} of any algebra is the  part of the algebra which
both commutes and associates with everything, i.e., those nuclear
elements commuting with all other elements:
\[
\cent(A) := \{c \in \caln uc(A) \mid [c,A] = 0\}.
\]

\item
Recall that for an associative algebra $A,$ the Jordan algebra $A^+$
is defined by $x\cdot y=\half(x y+y x)$.
Any subalgebra of a Jordan algebra of
type $A^+$ is called {\it special}.
\end{enumerate}
\end{Defs}

\begin{thm}[Shirshov-Cohn Theorem, Theorem 10, p.~48 in \cite{Jacobson}]\label{thm sc}
Any Jordan algebra over a field $\ff$ of characteristic not $2$ $($with $\e)$ generated by
two elements $($and $\e)$ is special.
\end{thm}

\begin{notation}\label{not main}
From now on we fix two {\it distinct} idempotents $a,b$ in the algebra
$A$ over a unital commutative ring $\ff$ ($A$ will be either associative,
or $A=J$ a Jordan algebra, and then $\ff$ is a field of characteristic not $2$).
We assume that $A$ is generated by $a$ and $b$ as an algebra over $\ff$
(but we do not assume that $A$ is unital).
Let
\[
\gs:=(a-b)^2=a+b-(a b+b a).
\]
\end{notation}

We need a few computations.

\begin{lemma}\label{lem assoc}
Assume that $A$ is associative (so multiplication in $A$ is denoted: $x y$).
\begin{enumerate}
\item
$\gs a=a-a b a=a\gs.$

\item
$\gs b=b-b a b=b\gs.$

\item
$a b a=(\e-\gs) a$ and $b a b=(\e-\gs) b.$

\item
$(\e-a) (\e-b) (\e-a)=(\e-\gs) (\e-a)$ and  $(\e-b) (\e-a) (\e-b)=(\e-\gs) (\e-b).$
\end{enumerate}
\end{lemma}
\begin{proof}
We have
$\gs a=(a+b-a b-b a) a=a+b a-a b
a-b a=a-a b a$.  Also $a \gs=a (a+b-a b-b a)=a+a b-a b-a
b a=a-a b a.$  Hence (1) holds.  Part (2) holds by symmetry.  Part (3) follows
from (1) and (2).  For part (4) notice that $x:=\e-a$ and $y:=\e-b$ are
idempotents in $A^{(1)}$ and $(x-y)^2=\gs$.  Hence, as in (3),
we get (4).
\end{proof}

\begin{lemma}\label{lem assoc comm}
Assume that $A$ is associative and commutative.  Then
\begin{enumerate}
\item
$\gs(a-b)=a-b,$ in particular $\gs^2=\gs;$

\item
$\gs a b=0;$

\item
$A=\ff\gs a\oplus\ff\gs b\oplus\ff a b$.
\end{enumerate}
\end{lemma}
\begin{proof}
We use Lemma \ref{lem assoc}.  We have $a b=a b a=a-\gs a$.
Similarly $b a=b-\gs b$.  Hence $a-\gs a=b-\gs b,$ so $a-b=\gs(a-b),$
and the first part of (1) holds.
Multiplying by $a-b$ we get (1).
Also $a b=(a b) b=a b-\gs a b,$ so $\gs a b=0,$
and   (2) holds.

Let $W$ be the $\ff$-linear combination of $\gs a,\, \gs b$ and $a b$.
Then\linebreak $a=ab+\gs a\in W$ and similarly $b\in W$.  Clearly, by (2), $W$ is closed
under multiplication, so $W=A$.  Suppose $\ga(\gs a)+\gb(\gs b)+\gc(a b)=0,$
with $\ga,\gb,\gc\in\ff$.
Multiplying by $\gs a$ and using (1) and (2) we get that $\ga(\gs a)=0$.
Similarly $\gb(\gs b)=0,$ and then $\gc(a b)=0$.  Also,
by (2), the sum is a direct sum of ideals, so (3) holds.
\end{proof}

In the next lemma, by a simple ring $R$
we mean a ring (not necessarily unital) such
that  $R^2\ne 0$ and the only proper  ideal
of $R$ is $\{0\}$.  This lemma is well known.
We include a proof for the convenience
of the reader.

\begin{lemma}\label{lem simple}
Let $R$ be a simple ring that satisfies a polynomial identity. Then
$R\cong M_n(D),$ for some division ring $D$.  In particular $R$ is
unital.
\end{lemma}
\begin{proof}
We show that $R$ is contained as an ideal
in a unital primitive ring $S$ that satisfies
a (multilinear) polynomial identity.
By Kaplansky's Theorem \cite[Theorem 23.31]{Row08}, $S$ is a simple ring which is finite
dimensional over its center (which is a field).
Since $S$ is simple and since $R$ is an ideal of $S$ we have $R=S$.
By Wedderburn's Theorem $R\cong M_n(D)$ as asserted.

If $R$ is unital, take $S=R$
(a simple unital ring is primitive).
So  suppose $R$ is not unital.  Now
$R$ is an algebra over the integers and we
let  $R^{(1)}$ be the ring defined above (adjoining
an identity $\e$ to $R$).  We identify $R$ with the ideal $\{(0,r)\mid r\in R\}$.
Consider the Jacobson Radical
$J(R^{(1)})$.  Since $R$ is an ideal of $R^{(1)}$ we have
$J(R)=J(R^{(1)})\cap R$ (\cite[Theorem 1.2.5]{herstein}).
Hence if $J(R^{(1)})\supseteq R,$
then $J(R)=R$ (i.e.~$R$ is a radical ring).
However, by \cite[Theorem 4.2]{Ja}, since
$R$ satisfies a polynomial identity, $J(R)\ne R$.
Since $R$ is simple, we see that $R\cap J(R^{(1)})=\{0\}$.

Let $S:=R^{(1)}/J(R^{(1)})$.  Then $R$ embeds in $S$ and
we consider $R$ as a subring of $S$.  Since $J(S)=\{0\},$
and since $J(S)$ is the intersection of all primitive
ideals of $S,$ there exists a primitive ideal $P$ of $S$
that does not contain $R,$ and hence intersects $R$ in $\{0\}$.
Replacing $S$ by $S/P$  we may assume
that $S$ is primitive.  Now since $R$ satisfies a polynomial
identity, it satisfies a multilinear polynomial identity
(\cite[Lemma 6.2.4]{herstein}).
Since $S$ is a central extension of $R,$ \cite[Proposition 23.8(i)]{Row08}
shows that $S$ satisfies a multilinear polynomial identity,
so we are done.
\end{proof}

\begin{proof}[{\bf Proof of Theorem \ref{thm assoc}}]
(1):\quad This follows from Lemma \ref{lem assoc}(1\&2).
\medskip

\noindent
(2): Let
\[
V = \ff[\gs]\gs + \ff[\gs]a  +\ff[\gs]b+  \ff[\gs]ab,
\]
the set of $\ff[\gs]$ linear combinations
of $\gs, a, b$ and $a b$.
We show that $V$ is a
subalgebra of $A$.  Since it contains $a$ and $b,$ this will show
that $A=V$.

Since $\gs\in {\rm Cent}(A)$   to show that $V$ is a subalgebra of $A$ it
suffices to show that $b a,\ a b a,\ b a
b\in V$, but this follows from Lemma \ref{lem assoc}, and from the fact that
$b a=-\gs+a+b-a b$.
The last part of (2) follows from \cite[Proposition 23.11]{Row08}.
\medskip

\noindent
(3):\quad
Suppose $\gs=\e$.  By Lemma \ref{lem assoc}(3),
$a b a=b a b=0$. If $a b=b a=0$
then it is easy to check that $A=\ff a\oplus\ff b$. Suppose $a b\ne 0$.  Then $(a b)^2=0$
and we see that $\ff(a b)$ is closed under multiplication by $a$ and $b$
from both sides, so (3) holds.
\medskip

\noindent
(4):\quad
Suppose $\gs=0$.  Then, by Lemma \ref{lem assoc}(3), $a b a=a$ and $b a b=b$.
Let $x:=a (\e-b), y:=b(\e-a)$ and $I:=\ff x+\ff y$.  If $x=y=0,$  then since $a-b\ne 0,$ we see that $\ff(a-b)$
is a nontrivial square-zero ideal.  Assume  $x\ne 0$.
Then $b x=b a (\e-b)=b a-b=-y,\ x b=0, a x=x$ and $x a=0$.
Similarly $a y, y a, b y, y b\in I$.  Also $x^2=y^2=x y=y x=0$.
Hence $I^2=\{0\}$ and (4) holds.
\medskip

\noindent
(5):\quad
In (3) we take $I = \ff a b+\ff b a$.
In (4) if  $x=y=0$ we take $I = \ff(a-b)$, since then the images of
$a$ and $b$ are equal.  So suppose $x\ne 0$.  Let $z=a b-b a$
and take $I=\ff (x+y)+\ff z$.
Note that $a b$ and
$b a$ are idempotents and since $v w v=v,$ for $\{v,w\}=\{a,b\},$
we see that $z^2=a b-a+b a-b=-x-y$.
It is easy to check that $z x=x+y,\, x z=0,\, z y=-x-y$
and $y z=0$.  Hence $I^2=\ff(x+y)$.  Now, by the above, and by the proof of (4), $I(x+y)=(x+y)I=\{0\}$.
Hence $I^3=\{0\},$ and $A/I$ is abelian.
\medskip

\noindent
(6):\quad
Set
$e_{1,1} = a,\ \ e_{1,2} =   a b (\e-a),\ \ e_{2,1} =
  (\gs (\e-\gs))^{-1}(\e-a) b a$ and $e_{2,2} = \e-a$.
Then
\[
e_{1,1} e_{i,j}=\gd_{1,i}e_{1,j}\quad\text{and}\quad e_{2,2} e_{i,j}=\gd_{2,i}e_{2,j}.
\]
Also,
\[
e_{1,2} e_{1,1}=0=e_{1,2} e_{1,2}\text{ and }e_{2,1} e_{2,2}=0=e_{2,1} e_{2,1}.
\]
Next, using Lemma \ref{lem assoc}(1\&3),
\begin{gather*}
\gs (\e-\gs)  e_{1,2} e_{2,1}  = a  b  (\e-a) b  a =  a  b
a - a  b  a  b  a \\
 = a  b  a -  a
 b  (a - \gs a) = \gs a  b  a = \gs
  (\e-\gs)  a,
\end{gather*}
Since $\gs(\e-\gs)$ is invertible in $A$ we get $e_{1,2} e_{2,1} = a = e_{1,1}$.

Similarly, using Lemma \ref{lem assoc}(3\&4),
\begin{gather*}
\gs (\e-\gs)   e_{2,1} e_{1,2}  =
(\e-a)  b  a b   (\e-a) =  (\e-a)  ( b-\gs b)
 (\e-a) \\  = (\e-\gs) (\e-a) (\e-b)(\e-a)=\gs (\e-\gs) (\e-a),
\end{gather*}
Since $\gs(\e-\gs)$ is invertible in $A$ we get $e_{2,1}  e_{1,2} = e_{2,2}.$ Thus the $e_{i,j}$ are
$2\times 2$ matrix units generating $A$ over $\ff[\gs]$ (see Definition 13.3 in \cite{Row08}).
By  \cite[Proposition~13.9]{Row08},  $A\cong M_2(R),$ where $R=a A a$.  Since $a b a=(\e-\gs)a$
(Lemma \ref{lem assoc}(3)), and by part (2), $R=\ff[\gs]a$.
Note now that $\ff[\gs]a\cong \ff[\gs],$ because if $\ga a=0,$ for some $\ga\in\ff[\gs],$
then $e_{2,1} \ga a e_{1,2}=\ga e_{2,2}=\ga (\e-a)=0,$ and then $\ga=0$.
Hence the map $\ga\mapsto\ga a$ is an isomorphism $\ff[\gs]\to R$ and (6) holds.
\medskip

\noindent
(7):\quad Suppose
next that $A$ is simple.  We may assume without loss that $a\ne 0$.  If $A$ is commutative, then $A=Aa,$ and $a=\e$
is the identity of $A$.  Thus $A$ is a field so $b=0$ and it follows that
$A=\ff\e$.

Suppose that $A$ is not commutative.
If $A$ is a division ring, then $\{a,b\}=\{\e, 0\}$
and $A=\ff\e$ is commutative, a contradiction.

By (2), $A$
satisfies a multilinear polynomial identity.
By Lemma \ref{lem simple}, $\e\in A,$
and $A\cong M_n(D)$ for some division ring $D$.
Let $\kk:={\rm Cent}(A)$.  Then $\kk$ is a field,
and by (2) the dimension of $A$ over $\kk$ is
at most $4$.  Since this dimension is a square which is not $1,$
it is $4$.

Hence  $\{\gs, a, b, a b\}$ are linearly
independent over $\kk$.  It follows that  $\{\gs, a, b, a b\}$ are linearly
independent over $\ff[\gs]\subseteq \kk$.  Now if $\gs$ is transcendental over $\ff,$
then $\ff[\gs]$ has a proper non-trivial ideal $I$.  And then
$I\gs+Ia+Ib+I(a b)$ would be a proper nontrivial ideal of $A,$
a contradiction.  Hence $\ff[\gs]$ is a field.
Since $A$ is simple and contains idempotents, $\gs\notin\{0,\e\},$
by (3) and (4).  Hence $\gs-\gs^2$ is invertible in $A,$
so $A\cong M_2(\ff[\gs])$ by (6).
\medskip

\noindent
(8):\quad
Suppose that $*$ is an involution on $A$.  Assume that
$\ga_{\gs}\gs +\ga_a a+\ga_b b+\ga_{a b}(a b)=0$.
Then also $\ga_{\gs}\gs +\ga_a a+\ga_b b+\ga_{a b}(b a)=0$.
Subtracting we get $\ga_{a b}(b a-a b)=0$.  Thus
condition $(i)$ of (8) holds.

Suppose condition $(i)$ of (8) holds, and assume that
\[
\ga_{\gs}\gs +\ga_a a+\ga_b b+\ga_{a b}(a b)=0.
\]
Then
\begin{gather*}
\ga_{\gs}\gs +\ga_a a+\ga_b b+\ga_{a b}(b a)=\\
\ga_{\gs}\gs +\ga_a a+\ga_b b+\ga_{a b}(a b)+\ga_{a b}(b a-a b)=\ga_{a b}(b a-a b).
\end{gather*}
By condition $(i),\ \ga_{a b}(b a-a b)=0,$ so $\ga_{\gs}\gs +\ga_a a+\ga_b b+\ga_{a b}(b a)=0$.
This shows that $*$ is well defined, and it is easy to check that it is an involution on $A$.

For the last part of (8), see Remark \ref{rem thm 1.1}(3) below
and note that $\ff[\gs]$ is a field, and $A$ is $4$-dimensional over $\ff[\gs]$.
\end{proof}

\begin{remarks}\label{rem thm 1.1}
\begin{enumerate}
\item
By \cite[Theorem 4]{L}, the converse of Theorem \ref{thm assoc}(3)
also holds, namely if $A=M_2(\kk)$ where $\kk$ is a finite simple field
extension of $\ff,$ then $A$ is generated over $\ff$ by two idempotents,
except in the case where $\kk=\ff_2,$ the field of two elements.

\item
Suppose that $a b=0$.  Then, by Lemma \ref{lem assoc}(1\&2),
$\gs=\e,$ and then $b a=-\e+a+b$.  Note that $\e(b a-a b)=b a,$ is
not necessarily $0,$ so it may happen that $*$ of Theorem \ref{thm
assoc}(8) is not an involution on $A$.

\item
Of course if $\{\gs, a, b, a b\}$ are independent over $\ff[\gs]$
then $*$ of Theorem \ref{thm assoc}(8) is an involution on $A$.
\end{enumerate}
\end{remarks}

\begin{proof}[{\bf Proof of Theorem \ref{thm jordan}}]
By the Shirshov-Cohn theorem, $J$ is a special algebra contained in
$A^+,$ where $A$ is an associative algebra generated over $\ff$ by
the idempotents $a$ and $b$.
\medskip

\noindent
(1):\quad
This follows immediately from Theorem \ref{thm assoc}(1).
\medskip

\noindent
(2):\quad
By (1), and since $a\cdot b=-\half(\gs+a+b),$  it follows that
the set of $\ff[\gs]$-linear combinations of $\gs, a, b$
is closed in $J$ under multiplication, and hence it is equal to $J$.
\medskip

\noindent
(3):\quad
Assume that $J$ is simple.  If $A$ is commutative,
then $J=A$ is a field, so $J=\ff$.  So assume that
$A$ is not commutative.  Let $I$ be a maximal ideal of $A$
not containing $J$.  Since $J$ is simple, $J\cap I=\{0\}$.
Hence we may replace $A$ with the simple associative
algebra $A/I$. Hence we may assume that
$A$ is simple. By Theorem \ref{thm assoc}(8), $*$ is an
involution on $A,$ and one easily checks
that $J=\calh(A,*)$.
\end{proof}

\section{Some additional results for the case where $\ff$ is a field
and $A$ is associative}

In this section we continue with Notation \ref{not main}.
We further assume that $A$ is  associative and that $\ff$
is a field.

\begin{prop}\label{prop alg or free}
Exactly one of the following holds:
\begin{enumerate}
\item[(a)]
$A$ is finite dimensional over $\ff,$
and $\gs$ is algebraic over $\ff$ (i.e.~it satisfies a polynomial in $\ff[\gl]$), or

\item[(t)]
$A$ is infinite dimensional over $\ff$ and $\gs$ is transcendental over $\ff$.
In this case
$A$ is isomorphic to the semigroup algebra of the free product $\lan a\ran*\lan b\ran$
of the one-element semigroups $\lan a\ran,\, \lan b\ran$.
\end{enumerate}
\end{prop}
\begin{proof}
If $A$ is finite dimensional over $\ff,$ then $(a)$ holds,
while if $A$ is infinite dimensional over $\ff$ then,
by \cite[Proposition 2]{L}, $A$ is as in $(t)$.

Suppose $A$ is as in $(t)$.  Then a direct and easy computation,
based on the leading term starting with $ab$ (or $ba$), shows that
$\gs=a+b-ab-ba$ cannot satisfy a polynomial over $\ff$.
\end{proof}

\begin{lemma}\label{lem assocgen}
Let $g_1, g_2\in\ff[\gl]$ be relatively prime polynomials such that
$g_1[\gs]g_2[\gs]=0$. Then $A\cong A/Ag_1[\gs]\times A/Ag_2[\gs].$
\end{lemma}
\begin{proof}
This follows from the Chinese Remainder Theorem, whose argument we review
since we need it for algebras without $\e$. First note that
$Ag_i[\gs]$ is an ideal of $A$ since $g_i[\gs]$ is central in $A^{(1)}$.

Now if $r \in Ag_1[\gs] \cap Ag_2[\gs]$ then $rg_2[\gs] = rg_1[\gs] = 0,$ so
writing
\[
a_1[\gl]g_1[\gl] + a_2[\gl]g_2[\gl]  = 1,
\]
for $a_1,a_2 \in \ff[\la]$, we see
that $r = ra_1[\gs]g_1[\gs] + ra_2[\gs]g_2[\gs] = 0,$
implying\linebreak $A \hookrightarrow A/Ag_1[\gs] \times  A/Ag_2[\gs]$. On the other hand,
for any $r_1 + Ag_1[\gs] \in A/Ag_1[\gs]$ and $r_2 + Ag_2[\gs] \in A/Ag_2[\gs]$ we take
\[
r = r_1a_2[\gs]g_2[\gs] + r_2a_1[\gs]g_1[\gs],
\]
 and note that
 \[
r   + Ag_1[\gs]   = r_1a_2[\gs]g_2[\gs] + Ag_1[\gs] = r_1(1 -
 a_1[\gs]g_1[\gs])+ Ag_1[\gs]   = r_1+ Ag_1[\gs]
\]
and likewise $r   + Ag_2[\gs]   = r_2 + Ag_2[\gs].$
\end{proof}

\begin{prop}\label{prop assocmain}
Let $g[\gl]\in\ff[\gl]$ be irreducible, let $k\ge 1$ and
let $h[\gl]=g[\gl]^k$.  Suppose that $Ah[\gs]\ne A,$ and let
$\widebar{A}=A/Ah[\gs]$. Let
$\bar\gs$ be the image
of $\gs$ in $A/Ah[\gs]$.  Then
\begin{enumerate}
\item
If $g[\gl]=\gl,$ then $\bar\gs$ is nilpotent in $\widebar{A},$ and there exists a nilpotent ideal $\widebar{J}$
in $\widebar{A}$ such that $\widebar{A}/\widebar{J}$ is abelian.

\item
If $g[\gl]\ne\gl,$ then $\widebar{A}$ is unital, furthermore $\bar\gs$ is
invertible in $\widebar{A}$.  Denote by $\bar e$ the identity
element of $\widebar{A}$. Then $h[\bar\gs]=\bar 0$ (where we substitute
$1$ by $\bar e$ in $h[\bar\gs]$).

\item
If $g[\gl]=\gl-1,$ then $\bar\gs-\bar e$ is nilpotent in $\widebar{A}$
and there exists a nilpotent ideal $\widebar{J}$
in $\widebar{A}$ such that $\widebar{A}/\widebar{J}$ is abelian.

\item
If $g[\gl]$  is relatively prime to $\gl(\gl-1),$ then, by (2),
$\widebar{A}$ is unital, and $\widebar{A}\cong M_2(\ff[\bar\gs])$.
\end{enumerate}
\end{prop}
\begin{proof}
(1):\quad Suppose that $g[\gl]=\gl$.  Then $Ah[\gs]=A\gs^k,$
so the ideal $\widebar{I}:=A^{(1)}\gs/A\gs^k$ is a nilpotent ideal in $\widebar{A}$. Indeed
\[
\left(A^{(1)}\gs\right)^{k+1}=\left(A^{(1)}\gs\right)A^{(1)}\gs^k\subseteq AA^{(1)}\gs^k=A\gs^k.
\]
Also in the algebra $\widebar{A}/\widebar{I}\cong A/A^{(1)}\gs$ the image of $\gs$
is $0$.  Hence by \linebreak Theorem \ref{thm assoc}(5), the algebra $\widebar{A}/\widebar{I}$ has
a nilpotent ideal over which it is commutative.
Let $\widebar{J}$ be the preimage in $\widebar{A}$ of that
ideal.  Then $\widebar{J}$ is the ideal whose existence is asserted in (1).
\medskip

\noindent
(2):\quad
Let
\[
\bar{ }\ \colon A\to\widebar{A}
\]
be the canonical homomorphism.
Write $h[\gl]=\ga\e+\gl q[\gl],$ with $0\ne\ga\in\ff$.
Then $h[\gs]=\ga\e+\gs q[\gs]$.  We have
\[
\bar 0=\widebar{h[\gs] x}=\ga\bar x+\widebar{\gs q[\gs]}\bar x,\text{ for all }x\in A.
\]
Let $e:=\gs (-\ga^{-1}q[\gs])\in A$.  Then we see that $\bar e=\widebar{\gs (-\ga^{-1}q[\gs])}$
is the identity element of $\widebar{A}$
and $\widebar{A}$ is unital.  Further
\[
\bar e=\widebar{(a-b)}\widebar{(a-b)(-\ga^{-1}q[\gs])}
\]
(recall that $\gs=(a-b)^2$), and we see that $\widebar{(a-b)}$
is invertible in $\widebar{A},$ so also $\bar\gs$ is invertible
in $\widebar{A}$.

Finally we have $h[\bar\gs]=h[\bar\gs]\bar e=\widebar{h[\gs]e}=\bar 0,$
so the last part of (2) holds.
\medskip

\noindent
(3):\quad
Suppose that $g[\gl]=\gl-1$.  Then the ideal $\widebar{I}:=A(\gs-\e)/Ah[\gs]$
is a nilpotent ideal in $\widebar{A}$ and the image of $\gs$ in $\widebar{A}/\widebar{I}$
is the identity of this algebra.  As in (1) we can apply Theorem \ref{thm assoc}(5) to obtain the ideal $\widebar{J},$
and (3) holds.
\medskip

\noindent
(4):\quad
Assume now that $g[\gl]$ is relatively prime to $\gl(\gl-1)$.
We show that $\bar\gs-\bar\gs^2$ is invertible in $\widebar{A}$.
Let $u[\gl], v[\gl]\in\ff[\gl]$ such that
\[
u[\gl](\gl-1)+v[\gl]h[\gl]=\e.
\]
Multiplying by $\gl$ we get that $u[\gl]\gl(\gl-1)+v[\gl]\gl h[\gl]=\gl$.
Substituting $\gs$ for $\gl$ we see that $u[\gs]\gs(\gs-\e)+v[\gs]\gs h[\gs]=\gs$.
Hence $\widebar{u[\gs](\gs^2-\gs)}=\bar\gs$.  Since $\bar\gs$ is invertible in
$\widebar{A}$ we see that $\bar\gs^2-\bar\gs$ is invertible in $\widebar{A}$.
Now Theorem \ref{thm assoc}(6) completes the proof of (4).
\end{proof}

\begin{notation}
If $B$ is an algebra generated by two idempotents $e$ and $f,$
we denote by $\gs_B:=(e-f)^2$ (here $e$ and $f$ are understood
from the context).
\end{notation}

As a corollary to Proposition \ref{prop assocmain} we get the following theorem, which handles the case where
$\gs$ is algebraic over $\ff$:
\begin{thm}\label{thm assocmain}
Suppose that $\gs$ is algebraic over $\ff$.
Then $A$ is a direct product of algebras $B_0\times B_1\times\dots\times B_m$
such that if we denote by $a_{B_i}, b_{B_i}$ the image of $a, b$ in $B_i,$
then we have
\begin{enumerate}
\item
$B_i$ is generated by the idempotents $a_{B_i}, b_{B_i}.$

\item
$B_i$ is unital for $i\ge 1$.

\item
$B_0=0,$ or $\gs_{B_0}$ is nilpotent and $B_0$ contains a nilpotent ideal
$J_0$ such that $B_0/J_0$ is commutative.

\item
$B_1=0,$ or $h[\gs_{B_1}]=0,$ where $h[\gl]=(\gl-1)^k,\, k\ge 1$.  Furthermore
$B_1$  contains a nilpotent ideal
$J_1$ such that $B_1/J_1$ is commutative.

\item
For $i\ge 2,$ we have $h[\gs_{B_i}]=0,$ where $h[\gl]=g[\gl]^k,\, k\ge 1,$ and
where $g[\gl]\in\ff[\gl]$ is an irreducible polynomial relatively prime to $\gl(\gl-1)$.  Furthermore
$B_i\cong M_2(\ff[\gs_{B_i}])$.
\end{enumerate}
\end{thm}
\begin{proof}
Write the monic minimal polynomial $m[\gl]$ of $\gs$ over $\ff$ as
\[
m[\gl]=\gl^{k_1}(\gl-1)^{k_2}g_3[\gl]^{k_3}\cdots g_m[\gl]^{k_m},
\]
with $g_3,\dots, g_m$ monic, irreducible, pairwise distinct and distinct from $\gl$ and $\gl-1,$
and where we allow $k_i=0,$ for $i=1$ or $i=2$.

By a repeated application of Lemma \ref{lem assocgen} we see that
\[
A\cong A/A\gs^{k_1}\times A/A(\gs-\e)^{k_2}\times A/A(g_3[\gs])^{k_3}\times\cdots\times A/A(g_m[\gs])^{k_m}.
\]
Now the theorem follows from Proposition \ref{prop assocmain}.
\end{proof}

\begin{lemma}
Assume that $\gs$ is transcendental over $\ff$.
Then
\begin{enumerate}
\item (Bergman~\cite[\S 12.2]{B})
$A$ is a free module over $\ff[\gs]$ with basis $\gs, a, b, ab;$

\item
there is an involution $*$ on $A$ as defined in
Theorem \ref{thm assoc}(8).
\end{enumerate}
\end{lemma}
\begin{proof}
First note that if $g[\gl]\in\ff[\gl]$ is a polynomial which is not a scalar
(i.e., $g\notin\ff$), then $Ag[\gs]\ne A$.  This follows from Proposition
\ref{prop alg or free}.

Suppose $f_1[\gs]\gs+f_2[\gs]a+f_3[\gs]b+f_4[\gs]ab=0$. Let
$g[\gl]\in\ff[\gl]$ be an irreducible polynomial prime to
$\gl(\gl-1)$ and to $f_1,\dots, f_4$. Consider $B:=A/Ag[\gs]$.  It
is a nontrivial algebra, so by Proposition \ref{prop assocmain}(4),
$B\cong M_2(\ff[\gs_B])$. Clearly  $\ff[\gs_B]$ is a field, so since
$B$ is $4$-dimensional over $\ff[\gs_B]$ and is spanned by the
images of $\gs, a, b, ab$ (Theorem \ref{thm assoc}(2)), we get a
contradiction. This shows (1) and (2) follows from (1) and Theorem
\ref{thm assoc}(8).
\end{proof}


\end{document}